\documentclass{amsart}
\usepackage{hyperref}
\usepackage{verbatim}
\usepackage{tikz}
\usepackage{tikz-cd}
\usepackage{mathdots}
\usepackage[alphabetic]{amsrefs}

\numberwithin{equation}{section}
\newtheorem{theorem}[equation]{Theorem}
\newtheorem*{theorem*}{Theorem}
\newtheorem{prop}[equation]{Proposition}
\newtheorem*{prop*}{Proposition}
\newtheorem{lemma}[equation]{Lemma}
\newtheorem{cor}[equation]{Corollary}

\newtheorem{conjecture}[equation]{Conjecture}
\theoremstyle{definition}
\newtheorem{example}[equation]{Example}
\newtheorem{definition}[equation]{Definition}
\newtheorem*{definition*}{Definition}
\newtheorem{remark}[equation]{Remark}

\usetikzlibrary{patterns,calc}

\newcommand{\bbP}{\mathbb{P}}

\newcommand{\bbZ}{\mathbb{Z}}
\newcommand{\bbN}{\mathbb{N}}

\newcommand{\field}{\mathbf{k}}
\newcommand{\faceof}{\prec}

\DeclareMathOperator{\pdim}{pdim}
\DeclareMathOperator{\depth}{depth}
\DeclareMathOperator{\vpdim}{vpdim}
\DeclareMathOperator{\lcm}{lcm}

\DeclareMathOperator{\length}{length}
\DeclareMathOperator{\Pic}{Pic}
\title{Virtual Resolutions of Monomial Ideals on Toric Varieties}
\author{Jay Yang}
\thanks{The author was supported by NSF DMS-1502553 and DMS-1745638.}

\begin{document}
\begin{abstract}
  We use cellular resolutions of monomial ideals to prove an analog of Hilbert's syzygy theorem for virtual resolutions of monomial ideals on smooth toric varieties.
\end{abstract}
\maketitle

\section{Introduction}

The theory of monomial ideals have provided a rich field of study for commutative algebra. These results rely on two fundamental features, one is that monomial ideals have combinatorial features that allow many invariants to be computed effectively, and second, the reduction to the initial ideal preserves or bounds these invariants. Thus many of the problems of commutative algebra can be reduced to, or at least bounded by the case of monomial ideals. Of particular interest to us is observation that the betti numbers of an ideal are bounded above by the betti numbers of its initial ideal, in particular the results of Bigatti--Hulett--Pardue~\cite{bigatti,hulett,pardue}.

The theory of virtual resolutions as described in \cite{virtual-res} by Berkesch, Erman, and Smith, provides an alternative description of a free resolution in the case of subvarieties of products of projective spaces or more generally, smooth toric varieties.

\begin{definition*}[\cite{virtual-res}*{Definition 1.1}]
  Given a smooth toric variety $X=X(\Sigma)$ and a $\Pic(X)$-graded module $M$, then a free complex $F$ of graded $\field[\Sigma]$-modules is a \emph{virtual resolution of $M$} if the corresponding complex $\widetilde{F}$ of vector bundles on $X$ is a resolution of $\widetilde{M}$.
\end{definition*}

Their original paper provides the following analogy to Hilbert's syzygy theorem which we attempt to generalize to the case of virtual resolutions on toric varieties. For $\mathbf{n}\in\bbN^s$ define $\bbP^{\mathbf{n}}=\bbP^{n_1}\times \cdots \times \bbP^{n_s}$.

\begin{prop*}[\cite{virtual-res}*{Proposition~1.2}]
  Given a smooth toric variety $X(\Sigma)$ with irrelevant ideal $B$.
  Every finitely generated $B$-saturated $S$-module on $\bbP^{\mathbf{n}}$ has a virtual resolution of length at most $\left|\mathbf{n}\right|$.
\end{prop*}

Given this, a natural question, stated as Question~7.5 in \cite{virtual-res} but phrased here as a conjecture, is to ask whether such a statement is true for arbitrary smooth toric varieties.

\begin{conjecture}
  \label{conj:virtual-hilbert-syzygy}
  Given a smooth toric variety $X(\Sigma)$ with irrelevant ideal $B$, and a $B$-saturated ideal $I$, $\vpdim S/I \leq n$, where \[\vpdim M = \min\left\{\length(F_{\bullet}) \mid F_{\bullet}\text{ a virtual resolution of } M\right\}.\]
\end{conjecture}

The proof of the previous proposition uses Beilinson's resolution of the diagonal, and is not immediately amenable to generalization to the toric case. Instead, we use the fact that a free resolution of $S/J$ is a virtual resolution of $S/I$ if $I=J:B^{\infty}$. With this in mind, we now state the main theorem of this paper.

\begin{theorem}
  \label{thm:main-theorem}
  Let $I$ be a non-irrelevant $B$-saturated monomial ideal on a complete simplicial $n$-dimensional normal toric variety. Then there exists a monomial ideal $J$ with $I=J:B^{\infty}$ with $\pdim S/J\leq n$.
\end{theorem}

The proof of this theorem is loosely inspired instead by Theorem~5.1 in \cite{virtual-res}.

\begin{theorem*}[\cite{virtual-res}*{Theorem~5.1}]
  If $Z\subset \bbP^{\mathbf{n}}$ is a zero dimensional punctual scheme with corresponding $B$-saturated ideal $I$ then there exists $\mathbf{a}\in \bbN^s$ with $a_r=0$ such that $\pdim S/(I\cap B^{\mathbf{a}}) \leq \left|\mathbf{n}\right|=n_1+\cdots+n_s$.
\end{theorem*}

As in the above theorem, it will turn out that the specific case that is of concern to us is when the virtual resolution is in fact a free resolution, but of a different ideal. This ideal will be constructed by starting with the original ideal intersected with a certain power of the irrelevant ideal, in this case the bracket power. Unfortunately as we will see in Lemma~\ref{lem:short}, this only gets us to $\pdim S/J\leq n+1$. Fortunately, there will be a reduction step that allows us to modify this to a new monomial ideal, with $\pdim S/J\leq n$.

The eventual goal is to have some degeneration theory to reduce Conjecture~\ref{conj:virtual-hilbert-syzygy} to Theorem~\ref{thm:main-theorem}. In particular, it would suffice to have an analogue of the following result.

\begin{theorem*}[\cite{combinatorial-ca}*{Theorem~8.29}]
  For $S$ a polynomial ring, and $I$ an ideal,
  \[\beta_{i,j}(S/I)\leq \beta_{i,j}(S/\operatorname{in}(I)).\]
\end{theorem*}

 Even in the absence of such a theorem, Theorem~\ref{thm:main-theorem} still provides support for Conjecture~\ref{conj:virtual-hilbert-syzygy}. However, significant work remains to prove the conjecture.

\section{Acknowledgments}

The author would like to thank his PhD advisor Daniel Erman, and his postdoc mentor Christine Berkesh for their helpful conversations and advice. In particular I'd like to thank Christine for pointing out that one can use monomial ideal labels in a labeled cell complex. I'd also like to thank Mike Loper for our conversations as I worked out the details of the paper. Finally, I'd like to thank the anonymous reviewer of this paper for their comments and for their patience as I rewrote a large part of the paper.

\section{Notations and Conventions}

For the most part, we use standard notation for fans and cones. However, for the purposes of this paper, we will regard a cone as the set of extremal rays. Then as is standard, a fan $\Sigma$ is a finite collection of strongly convex rational polyhedral cones such that every face of a cone in $\Sigma$ is in $\Sigma$ and the intersection of any two cones in $\Sigma$ is a face of both cones. The set $\Sigma(n)$ is the set of $n$-dimensional cones. For a ray $\tau$, we denote the corresponding variable in the Cox ring by $x_{\tau}\in \field[\Sigma]$. Similarly, for a collection of rays $\sigma$ let $x_{\sigma}:=\prod_{\tau\in \sigma}x_{\tau}$ and $x_{\widehat{\sigma}}:=\prod_{\tau\notin\sigma}x_{\tau}$.


Finally, For cells $F$,$G$ in a cell complex, we will use $F\faceof G$ to denote that $F$ is a face of $G$.
  
\section{Cellular Resolutions}

Our main tool in this paper will be cellular resolutions. The underlying concepts for cellular resolutions were first described in the case of simplicial complexes by Bayer, Peeva, and Sturmfels in their paper Monomial Resolutions~\cite{monomial-res}. This was later generalized to cell complexes by Bayer and Sturmfels in~\cite{cellular-res}. We however will need a slight generalization, which itself is a special case of a generalization described by Ezra Miller in \cite{miller-cm}, where we allow for the associated cell complex to be labeled by monomial ideals.

\begin{definition}
  A \emph{labeled cell complex} $(\Delta,I_{\bullet})$ is a cell complex $\Delta$ together with a collection of monomial ideals $\{I_{F}\}_{F\in\Delta}$ with $I_{F}\subset I_{G}$ for $G\faceof F$.
\end{definition}

The following lemma is modeled after \cite[Definition~6.7]{cb-hypergeometric} and \cite[Definition~3.2]{miller-cm}.

\begin{lemma}
  \label{lem:cell-complex}
  Fix a polynomial ring $S=\field[x_1,\ldots,x_n]$ with the fine grading by $\bbZ^n$. Let $\Delta$ be a labeled cell complex with $\mathcal{E}(F,G)$ the attaching degree from the cell $G$ onto the cell $F$. then for $\alpha\in\bbZ^n$ define the subcomplex of $\Delta$ whose cells are given by $\Delta_{\alpha}=\left\{F\, |\, (I_{F})_{\alpha}\neq 0\right\}$. Then the chain complex
  \[C_{\Delta}:0\rightarrow \bigoplus_{\dim F=n} I_{F}\rightarrow \cdots \rightarrow \bigoplus_{\dim F=1} I_{F}\rightarrow \bigoplus_{\dim F=0} I_{F}\rightarrow 0,\]
  with boundary maps
  \[(a_{F})_{\dim F=i}\mapsto \left(\sum_{\dim F=i} \mathcal{E}(F,G)a_{F}\right)_{\dim G =  i-1},\]
  has homology given by the following formula
  \[H_{i}(C_{\Delta})_m = H_{i}(\Delta_{m},\field).\]
\end{lemma}
\begin{proof}
  Consider a multidegree $\alpha\in\bbZ^n$, then, since $(I_{F})_{\alpha}=\field$ for $F\in \Delta_{\alpha}$ we get that the subcomplex $(C_{\Delta})_{\alpha}$ is given by the cellular chain complex on $\Delta_{\alpha}$ with coefficients in $\field$.
\end{proof}

\begin{remark}
  For a detailed introduction cell complexes, see ~\cite{hatcher}. In particular, the attaching degree mentioned in Lemma~\ref{lem:cell-complex} is the degree of the map given in \cite[Cellular Boundary Formula]{hatcher} and is the same coefficient that shows up in the usual cellular chain complex. 
  For the complexes constructed for this paper, the attaching degrees will always be $1$, $0$, or $-1$, with $\mathcal{E}(F,G)=\pm 1$ if $F\faceof G$ and $\mathcal{E}(F,G)=0$ and where the sign of $\mathcal{E}(F,G)$ is determined by the relative orientation of the faces $F$ and $G$ in the cell complex.
\end{remark}

Then as with cellular resolutions with monomial labels, Lemma~\ref{lem:cell-complex} implies that with an appropriate acyclicity condition, there exists a resolution of the ideal generated by the vertex labels.

\begin{cor}
  \label{cor:resolution}
  If $\Delta$ is a labeled cell complex and $\widetilde{H}_i(\Delta_{\alpha},\field)=0$ for all $\alpha$, then $C_{\Delta}$ is an complex with homology concentrated in homological degree $0$ and $H_0(C_{\Delta})=\sum_{\dim F = 0} I_{F}$.
\end{cor}
\begin{proof}
  Since $\Delta_{\alpha}$ is given by the set of cells $G$ such that $\alpha\in I_{G}$ and $I_G\subset I_{F}$ for $F\faceof G$, we have that $\Delta_{\alpha}\neq \emptyset$ if and only if there exists a cell $F$ with $\dim F=0$ and $x^\alpha\in I_{F}$. Thus by appling Lemma~\ref{lem:cell-complex} gives the following:
  \[H_i(C_\Delta)_{\alpha}=
  \begin{cases}
    \field & \alpha\in \sum_{\dim F = 0} I_{\sigma}\text{ and } i=0,\\
    0 & \text{otherwise}.
  \end{cases}\]
  This implies
  \[
    H_i(C_{\Delta})=\begin{cases}
    \sum_{\dim F = 0} I_{F} & i=0,\\
    0 & \text{otherwise}.
    \end{cases}\]
    Thus the complex $C_{\Delta}$ has homology concentrated in degree $0$, with $H_{0}(C_{\Delta})=\sum_{\dim F = 0} I_{F}$.
\end{proof}

\begin{remark}
  While in a strict homological sense, the complex in Corollary~\ref{cor:resolution} is a resolution, we avoid the term here to avoid confusion with both free resolutions and minimal resolutions.
\end{remark}

\begin{remark}
  \label{rem:induce}
  Given any cell complex $\Delta$ and a labeling of the vertices of $\Delta$ by monomial ideals, we can naturally extend this labeling to a labeling of all cells of $\Delta$ by simply defining the labeling on a cell to be the intersection of the labels on its vertices. Under such a labeling, the induced subcomplex $\Delta_{\alpha}$ is uniquely determined by the set of vertices it contains.
\end{remark}

So far, the construction has exactly mirrored the usual cellular resolution for monomial ideals. But one important difference is that there is no directly analogous minimality result. In the classical case, so long as no cell has the same label as one of its faces, the resulting resolution will be minimal. Since our complex is not a free resolution, to discuss minimality we must first pass to the total complex, but the resulting resolution will rarely be minimal.

\begin{example}
  Let $\Delta_1$ be the $1$-simplex, let $I_1 = (x^2,xy^2)$, $I_2 =(y^2,x^2y)$, and $I_{\left\{1,2\right\}}=(xy^2,x^2y)$, then $I=I_1+I_2 = (x^2,y^2)$.
  Now notice that $\pdim S/I_{\left\{1,2\right\}}=2$, and so the free resolution from the total complex has length at least $3$, but we know that $\pdim S/I=2< 3$.
\end{example}

\section{Short Virtual Resolutions via Bracket Powers}
\label{sec:bracket}
The second component of our technique is an observation that intersecting monomial ideals with a bracket power of the irrelevant ideal leads to simpler resolutions. The bracket power of an ideal denoted $I^{[k]}$ is given by the ideal generated by the image of $I$ under the ring homomorphism given by $x_i\mapsto x_i^{k}$. Conveniently, for monomial ideals this is simply given by the formula $\left<m_1,\ldots,m_s\right>^{[k]}=\left<m_1^k,\ldots,m_s^k\right>$.

\begin{lemma}
  \label{lem:short}
  If $I$ is a monomial ideal and $B$ is the irrelevant ideal of a $n$-dimensional complete normal toric variety, then $\pdim S/(I\cap B^{[k]})\leq n+1$.
\end{lemma}
\begin{proof}
  The irrelevant ideal $B$ is generated by the monomials corresponding to the complements of cones in $\Sigma$. In particular,
  \[B=\left<x_{\widehat{\sigma}}|\sigma \in \Sigma \right>=\sum_{\sigma\in\Sigma}\left<x_{\widehat{\sigma}}\right>.\]
  Thus for $k$ sufficiently large, all the generators of $I\cap B^{[k]}$ are of the form $x_{\widehat{\sigma}}^k\prod_{\tau\in \sigma}x_\tau^{a_{\tau}}$ with $a_{\tau}\leq k$. Then there exists a decomposition of $I\cap B^{[k]}=\sum_{\sigma\in\Sigma} I_{\sigma}$ with $I_{\sigma}=I\cap \left<x_{\widehat{\sigma}}^k\right>$. To start, consider $I_{\sigma\cap\sigma'}$.
  \begin{equation}
    \label{eq:intersect}
    I_{\sigma\cap\sigma'}=\left<x_{\widehat{\sigma\cap\sigma'}}^k\right>\cap I=\left(\left<x_{\widehat{\sigma}}^k\right>\cap I\right)\cap \left(\left<x_{\widehat{\sigma'}}^k\right>\cap I\right)=I_{\sigma}\cap I_{\sigma'}.
  \end{equation}
  
  Now consider the cell complex $\Delta$ that is dual to the poset of cones. In the case where $\Sigma$ is normal fan of a polytope, this corresponds to that polytope, but $\Sigma$ need not be the normal fan of a polytope for this cell complex to exist. Denote the cells of the complex by $[\sigma]$ for $\sigma\in \Sigma$. This is a labeled cell complex, with labels given by the $I_{\sigma}$, allowing us to apply Lemma~\ref{lem:cell-complex}. This gives the following complex:

  \begin{equation}
    \label{eq:complex}
    0\rightarrow \bigoplus_{\dim \sigma=0} I_{\sigma}\rightarrow\cdots \rightarrow \bigoplus_{\dim \sigma=n-1} I_{\sigma}\rightarrow\bigoplus_{\dim \sigma=n} I_{\sigma}\rightarrow I.
  \end{equation}

  To show this complex is exact it suffices to show that all subcomplexes associated to a monomial are contractable. Fix some monomial $m$, then since each of the $I_{\sigma}$ are divisible by $x_{\widehat{\sigma}}^k$, if $\sigma\in\Delta_m$, then $x_{\widehat{\sigma}}^k\mid m$.

  Note that (\ref{eq:intersect}) implies if $[\sigma],[\sigma']\in\Delta_m$, then $[\sigma\cap\sigma']\in \Delta_{m}$. Now let $\tau=\bigcap_{[\sigma]\in\Delta_m} \sigma$, then $[\tau]\in\Delta_m$. Since $\tau\subset \sigma$ for $[\sigma]\in \Delta_m$, we have $[\sigma]\faceof [\tau]$. As such $\Delta_m\subset \overline{[\tau]}$. Thus $\Delta_m=\overline{[\tau]}$. In $\Delta$, the closure of any cell is contractible, thus $\Delta_m$ is contractible. Thus the complex in~(\ref{eq:complex}) is exact.
  
  Finally, since $I_{\sigma}$ is in essence in $\dim(\sigma)$ variables, for $\dim(\sigma)=k$, we have $\pdim S/I_{\sigma}\leq k$ by Hilbert's syzygy theorem and thus $\pdim I_{\sigma}\leq \max(k-1,0)$. Then by taking the minimal free resolutions of $I_{\sigma}$, we get the following double complex of free modules:
  
  \begin{center}
    \begin{tikzcd}[sep=small]
  & & & & & 0 \arrow[d]\\
  & & & & & \displaystyle \bigoplus_{\dim \sigma=n} F_{\sigma,n-1}\arrow[d]\\
  & & & & \iddots & \vdots \arrow[d]\\
  & & & 0 \arrow[r] \arrow[d] & \cdots \arrow[r] & \displaystyle \bigoplus_{\dim \sigma=n} F_{\sigma,2} \arrow[d]  \\
  & 0 \arrow[r] \arrow[d] & 0 \arrow[r] \arrow[d]  & \displaystyle \bigoplus_{\dim \sigma=2} F_{\sigma,1} \arrow[r] \arrow[d] & \cdots \arrow[r] &  \displaystyle \bigoplus_{\dim \sigma=n}F_{\sigma,1} \arrow[d] \\
  0\arrow[r] & \displaystyle \bigoplus_{\dim \sigma=0} F_{\sigma,0}\arrow[r] & \displaystyle \bigoplus_{\dim \sigma=1} F_{\sigma,0} \arrow[r] & \displaystyle \bigoplus_{\dim \sigma=2} F_{\sigma,0} \arrow[r] & \cdots \arrow[r] &  \displaystyle \bigoplus_{\dim \sigma=n}F_{\sigma,0}. 
    \end{tikzcd}
  \end{center}
  The columns of this complex are the direct sums of the free resolutions of the $I_{\sigma}$ terms in (\ref{eq:complex}) and the rows are the maps induced from the complex in (\ref{eq:complex}). Now if we take the total complex of the above complex, we get a free resolution of $I\cap B^{[k]}=\sum_{\sigma\in \Sigma(n)} I_{\sigma}$. This gives $\pdim I\cap B^{[k]}\leq n$ and thus $\pdim S/I\cap B^{[k]} \leq n+1$.
\end{proof}

\begin{example}
  \label{ex:labeling}
  Let $X$ be the toric variety $\bbP^2\times \bbP^1$, and label the vertices like this
  
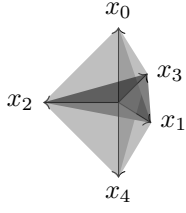
\begin{center}
  \begin{tikzpicture}
    \coordinate (o) at (0,0,0);
    \coordinate (x0) at (0,1,0);
    \coordinate (x1) at (0.7071,0,0.7071);
    \coordinate (x2) at (-1,0,0);
    \coordinate (x3) at (0,0,-1);
    \coordinate (x4) at (0,-1,0);
    
    \node[above] at (x0) {$x_0$};
    \node[right] at (x1) {$x_1$};
    \node[left] at (x2) {$x_2$};
    \node[right] at (x3) {$x_3$};
    \node[below] at (x4) {$x_4$};
    
    \draw[->] (o) -- (x0);
    \draw[->] (o) -- (x1);
    \draw[->] (o) -- (x2);
    \draw[->] (o) -- (x3);
    \draw[->] (o) -- (x4);

    \fill[opacity=0.5,gray] (o) -- (x0) -- (x1);
    \fill[opacity=0.5,gray] (o) -- (x0) -- (x2);
    \fill[opacity=0.5,gray] (o) -- (x0) -- (x3);

    \fill[opacity=0.5,black] (o) -- (x1) -- (x2);
    \fill[opacity=0.5,black] (o) -- (x2) -- (x3);
    \fill[opacity=0.5,black] (o) -- (x3) -- (x1);

    \fill[opacity=0.5,gray] (o) -- (x4) -- (x1);
    \fill[opacity=0.5,gray] (o) -- (x4) -- (x2);
    \fill[opacity=0.5,gray] (o) -- (x4) -- (x3);

    
    
  \end{tikzpicture}.
  \end{center}
  Now take the following monomial ideal
  \[I=\left<{x}_{2} {x}_{3},{x}_{1}^{4} {x}_{2}^{2} {x}_{4},{x}_{0}^{2} {x}_{1}^{4}
       {x}_{4},{x}_{1}^{5} {x}_{2}^{2},{x}_{0}^{2} {x}_{1}^{5},{x}_{1}^{4} {x}_{3}^{3} {x}_{4},{x}_{1}^{5}
       {x}_{3}^{3}\right>.\]
  Then
  \begin{equation*}
    \begin{split}
    I\cap B^{[6]}&=\\
    &x_3^6x_4^6\cdot \left<x_2,x_1^4\right> +\\
    &x_1^6x_4^6\cdot \left<x_2x_3,x_2^2,x_0^2,x_3^3\right> +\\
    &x_2^6x_4^6\cdot \left<x_3,x_1^4\right> +\\
    &x_0^6x_3^6\cdot \left<x_2,x_1^4x_4,x_1^5\right> +\\
    &x_0^6x_1^6\cdot \left<1\right> +\\
    &x_0^6x_2^6\cdot \left<x_3,x_1^4x_4,x_1^5\right>.
    \end{split}
  \end{equation*}

  This gives rise to the following labeled cell complex
  
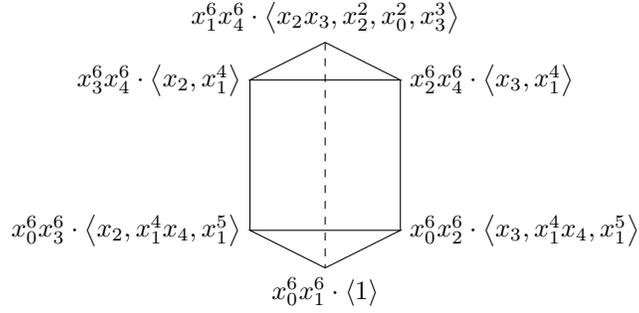
\begin{center}
  \begin{tikzpicture}
    \coordinate (o) at (0,0,0);
    \coordinate (x0) at (0,1,0);
    \coordinate (x1) at (0.7071,0,0.7071);
    \coordinate (x2) at (-1,0,0);
    \coordinate (x3) at (0,0,-1);
    \coordinate (x4) at (0,-1,0);

    \coordinate (v1) at (-1,1); 
    \coordinate (v2) at (1,1); 
    \coordinate (v3) at (0,1.5); 
    \coordinate (v4) at (-1,-1); 
    \coordinate (v5) at (1,-1); 
    \coordinate (v6) at (0,-1.5); 

    \node[left] at (v1) {$x_3^6x_4^6\cdot \left<x_2,x_1^4\right>$};
    \node[above] at (v3) {$x_1^6x_4^6\cdot \left<x_2x_3,x_2^2,x_0^2,x_3^3\right>$};
    \node[right] at (v2) {$x_2^6x_4^6\cdot \left<x_3,x_1^4\right>$};
    \node[left] at (v4) {$x_0^6x_3^6\cdot \left<x_2,x_1^4x_4,x_1^5\right>$};
    \node[below] at (v6) {$x_0^6x_1^6\cdot \left<1\right>$};
    \node[right] at (v5) {$x_0^6x_2^6\cdot \left<x_3,x_1^4x_4,x_1^5\right>$};
    
    \draw (v1)--(v2)--(v3)--cycle;
    \draw (v4)--(v5)--(v6)--cycle;

    \draw (v1)--(v4);
    \draw (v2)--(v5);
    \draw[dashed] (v3)--(v6);
  \end{tikzpicture}.
  \end{center}
\end{example}

\begin{prop}
  For $k\gg 0$, the total betti numbers $\beta_i$ of $I\cap B^{[k]}$ are independent of $k$.
\end{prop}

This result is similar in flavor to results of Mayes-Tang~\cite{mayes-tang-stability} and Whieldon~\cite{whieldon-stability}, which describe the stabilization of the shapes and decompositions of the betti table of the usual power of an ideal.

\begin{proof}
  Start with $I\cap B^{[k]}$ with $k$ larger than the largest degree of a generator in $I$. Then there is some polarization $J_k$ of $I\cap B^{[k]}$. Let $y_{i,j}$ denote the $i$-th variable added by polarization corresponding to the variable $x_{j}$. Now relate the polarization $J_{k}$ with the polarization $J_{k+1}$. Since $k$ is larger than the degree of any generator of $I$, every generator in $I\cap B^{[k]}$ is of the form $mx_{\widehat{\sigma}}^k$ for some cone $\sigma\in\Sigma$ and monomial $m$ with variables in $\sigma$. Thus after polarization, we get $\widetilde{m}\cdot x_{\widehat{\sigma}}\prod_{i=2}^{k}y_{i,\widehat{\sigma}}$, where $\widetilde{m}$ is the polarization of $m$.
  
  Define rings
  \begin{align*}
    I&\subset R,\\
    J_{k}&\subset S := R[\mathbf{y}_2,\ldots,\mathbf{y}_k],\\
    J_{k+1}&\subset S' := S[\mathbf{y}_{k+1}],
  \end{align*}
  where $\mathbf{y}_i$ represents the variables $y_{i,1},\ldots,y_{i,n}$.

  Since the generators of $J_{k}$ and $J_{k+1}$ differ simply by multiplication by the variables $y_{k+1,i}$, \[J_{k}=J_{k+1}/\left<y_{k+1,1}-1,\ldots,y_{k+1,n}-1\right>\subset S\cong S'/\left<y_{k+1,1}-1,\ldots,y_{k+1,n}-1\right>.\] Since $y_{k+1,1}-1,\ldots,y_{k+1,n}-1$ form a regular sequence, a free resolution of $J_{k+1}$ gives a (not necessarily minimal) free resolution of $J_{k}$.

  Then since the betti numbers of a monomial ideal are equal to those of its polarization, the total betti numbers of $I\cap B^{[k]}$ are non-decreasing with increasing $k$. But since $I_{\sigma} = I \cap \left<x_{\widehat{\sigma}}^k\right> = x_{\widehat{\sigma}}^k\cdot (I:x_{\widehat{\sigma}}^k)$, the total betti numbers of $I_{\sigma}$ are independent of $k$ and so the total betti numbers $I\cap B^{[k]}$ are bounded by the free resolution given by Lemma~\ref{lem:short}. As such, the total betti numbers must in fact stabilize.
\end{proof}

\section{A Shorter Resolution}

While the resolution constructed in Section~\ref{sec:bracket} is shorter than the minimal free resolution in general, it is not as short as the minimal resolutions of ideals on projective space or the short virtual resolutions on products of $\bbP^n$ given by \cite{virtual-res}. We can however make the following observation.

\begin{remark}
  Under the hypotheses of Lemma~\ref{lem:short}, the $(n+1)$-th total betti number of $S/I\cap B^{[k]}$, $\beta_{n+1}$, is at most $1$. In particular, the free resolution given by the proof of Lemma~\ref{lem:short} will have rank $1$ in homological degree $n+1$.
\end{remark}

This $(n+1)$-th betti number corresponds to the single top dimensional cell of the cell complex $\Delta$. So to reduce the length of the resolution, we will try to find a cellular resolution given by a cell complex of one dimension lower. As part of this, we will need to combine the labels on the previous cell complexes to give labels on the new cell complex.

For this let us first define the new cell complex on which we will provide labels.

\begin{definition}
  Given $\Delta$ as in the previous section, fix any ray $\tau$, and define $\widetilde{\Delta}$ to be the subcomplex of $\Delta$ where $[\sigma]\in\widetilde{\Delta}$ if $\sigma\cup\left\{\tau\right\}$ does not form a cone.
\end{definition}

The notation above does not mention the ray $\tau$, because while the choice of the ray $\tau$ affects the ideal and thus virtual resolution that this section constructs, the choice does not affect the conclusions of this paper.

Now for each $[\sigma]\in \widetilde{\Delta}$ we will define a label. To do this, for each $[\sigma]\in \widetilde{\Delta}$ we will define a subset of the set of cones $S(\sigma)$ as the set of cones who's defining rays are a subset of $\sigma\cup\left\{\tau\right\}$.
\[S(\sigma)=\left\{\gamma\in \Sigma\, |\, \gamma\subset \sigma\cup \left\{\tau\right\}\right\}.\]

Now we can define the label on $[\sigma]$ in $\widetilde{\Delta}$ as the following ideal:
\[J_{\sigma} = x_{\widehat{\sigma\cup\left\{\tau\right\}}}^k\cdot \widetilde{J}_{\sigma}\]
where
\begin{equation}
  \label{eq:j-tilde}
  \widetilde{J}_{\sigma} =\bigcap_{\gamma\in S(\sigma)} (I_{\gamma}:x_{\widehat{\gamma}}^\infty).
\end{equation}

\begin{remark}
  Since $I_{\gamma} = \left<x_{\widehat{\gamma}}^k\right>\cap I$, $I_{\gamma}:x_{\widehat{\gamma}}^\infty=I:x_{\widehat{\gamma}}^\infty$. Then since all generators of $I$ have degree in each variable of at most $k$, $I_{\gamma}:x_{\widehat{\gamma}}^k=I:x_{\widehat{\gamma}}^{\infty}$. Thus equation~\ref{eq:j-tilde} is equivalent to \[\widetilde{J}_{\sigma} =\bigcap_{\gamma\in S(\sigma)} (I:x_{\widehat{\gamma}}^k).\]

  Also if $\gamma'\subset \gamma$, then $I:x_{\widehat{\gamma}}^\infty\subset I:x_{\widehat{\gamma'}}^\infty$. So it in fact suffices to take the intersection over cones that are maximal in $S(\sigma)$.

\end{remark}

Now as before, we must show that these labels satisfy the appropriate conditions to give us a resolution of the correct length. We start by showing that the labels respect intersection of cones. An immediate consequence of this will be that for cells $[\sigma]\faceof [\sigma']$ we have $J_{\sigma'}\subset J_{\sigma}$ 

\begin{lemma}
  \label{lem:j-intersect}
  For cells $[\sigma],[\sigma']\in \widetilde{\Delta}$ the labels satisfy the following property: \[J_{\sigma\cap\sigma'}=J_{\sigma}\cap J_{\sigma'}.\]
\end{lemma}
\begin{proof}
  For convenience, we define the collection of rays $\omega= (\sigma\cap\sigma')\cup \left\{\tau\right\}$. For convenience, let $m=x_{\widehat{\omega}}$ which gives $J_{\sigma}\cap J_{\sigma'}\subset \left<m^k\right>$. Then it suffices to show \[(J_{\sigma}\cap J_{\sigma'}):m^k = \widetilde{J}_{\sigma\cap\sigma'}.\]

  Now expanding the left hand side, we get the following:
  \begin{align*}
    (J_{\sigma}\cap J_{\sigma'}):m^k &= (J_{\sigma}:m^k\cap J_{\sigma'}:m^k)\\
    &= (x_{\widehat{\sigma\cup\left\{\tau\right\}}}^k\widetilde{J}_{\sigma}):m^k\cap (x_{\widehat{\sigma'\cup\left\{\tau\right\}}}^k\widetilde{J}_{\sigma'}):m^k\\
    &=\widetilde{J}_{\sigma}: (m/x_{\widehat{\sigma\cup\left\{\tau\right\}}})^k\cap \widetilde{J}_{\sigma'}:(m/x_{\widehat{\sigma'\cup\left\{\tau\right\}}})^k\\
    &=\left(\bigcap_{\gamma\in S(\sigma)} (I:x_{\widehat{\gamma}}^\infty)\right): (m/x_{\sigma\cup\left\{\tau\right\}})^k\cap\left(\bigcap_{\gamma\in S(\sigma')} (I:x_{\widehat{\gamma}}^\infty)\right): (m/x_{\sigma'\cup\left\{\tau\right\}})^k\\
    &=\left(\bigcap_{\gamma\in S(\sigma)} \left(I:x_{\widehat{\gamma}}^\infty: (m/x_{\sigma\cup\left\{\tau\right\}})^k\right)\right)\cap \left(\bigcap_{\gamma\in S(\sigma')} \left(I:x_{\widehat{\gamma}}^\infty: (m/x_{\sigma'\cup\left\{\tau\right\}})^k\right)\right) 
  \end{align*}

  Again, $I$ is monomial and has only generators with degrees in each variable at most $k$, so we can simplify.

  \begin{align*}
    I:x_{\widehat{\gamma}}^{\infty}: (m/x_{\widehat{\sigma\cup\left\{\tau\right\}}})^k &= I:x_{\widehat{\gamma}}^\infty: (m/x_{\widehat{\sigma\cup\left\{\tau\right\}}})^\infty\\
    &= I:\lcm(x_{\widehat{\gamma}},m/x_{\widehat{\sigma\cup\left\{\tau\right\}}})^\infty\\
    &= I:x_{\widehat{\gamma\cap \omega}}^\infty
  \end{align*}

  The cone $\gamma\cap \omega$ is a cone in $S(\sigma\cap\sigma')$ and moreover, every cone in $S(\sigma\cap\sigma')$ can be written in this form, thus by combining with the previous work, we find that \[(J_{\sigma}\cap J_{\sigma'}):m^k = \bigcap_{\sigma\in S(\sigma\cap\sigma')}  I:x_{\widehat{\gamma}}^\infty = \widetilde{J}_{\sigma\cap\sigma'}.\qedhere\]
  \end{proof}

\begin{lemma}
  \label{lem:j-pdim}
  For $[\sigma]\in \widetilde{\Delta}$ we have $\pdim J_{\sigma}\leq \dim \sigma-1.$
\end{lemma}
\begin{proof}
  This reduces to showing that $\pdim \widetilde{J}_{\sigma}\leq \dim \sigma-1$. Furthermore, since the generators of $\widetilde{J}_{\sigma}$ can be expressed in terms of only those variables corresponding to rays in $\sigma\cup \left\{\tau\right\}$, we can view it as an ideal of $R=\field[\sigma\cup\left\{\tau\right\}]$ rather than $S=\field[\Sigma]$.
  
  If $\widetilde{J}_{\sigma}=R$, then the statement is trivial, otherwise, we can use the Auslander-Buchsbaum formula to bound the projective dimension.
  \begin{align*}
    \pdim \widetilde{J}_{\sigma} &= \pdim R/\widetilde{J}_{\sigma} - 1\\
    &\leq \depth R-\depth R/\widetilde{J}_{\sigma}-1\\
    &= \dim\sigma-\depth R/\widetilde{J}_{\sigma}.
  \end{align*}
  Thus it suffices to show that $R/\widetilde{J}_{\sigma}$ is depth at least $1$. This we can do this by exhibiting an explicit nonzerodivisor, namely\[m = \sum_{\rho\in S(\sigma)} x_{\rho}.\]

    Thus we wish to show that if $mf\in \widetilde{J}_{\sigma}$ then $f\in \widetilde{J}_{\sigma}$ for $x\in R$. Since $\widetilde{J}_{\sigma}=\bigcap_{\gamma\in S(\sigma)} I:x_{\widehat{\gamma}}^\infty$, it suffices to show for all $\gamma\in S(\sigma)$ that if $mf \in I:x_{\widehat{\gamma}}^\infty$ then $f\in I:x_{\widehat{\gamma}}^\infty$
    
    For this statement, let $\rho\in \sigma\cup\left\{\tau\right\}$ and $\rho\notin \gamma$, and fix a lex monomial order with $x_{\rho}$ the maximal variable. Then consider an element $f\notin I:x_{\widehat{\gamma}}^\infty$. Since our ideal is monomial we may assume that $f$ contains no terms in $I:x_{\widehat{\gamma}}^\infty$.

    Then by computing leading terms we find $LT(mf) = x_{\rho}LT(f)$. Since $x_{\rho}|x_{\widehat{\gamma}}$, it follows that $I:x_{\widehat{\gamma}}^\infty$ has no generators divisible by $x_{\rho}$. Furthermore, by assumption $LT(f)\notin I:x_{\widehat{\gamma}}^\infty$ so this implies $x_{\rho}LT(f)\notin I:x_{\widehat{\gamma}}^\infty$ and thus $mf\notin I:x_{\widehat{\gamma}}^\infty$. Thus $m$ is not a zero-divisor and since $m$ is non-zero, it is a nonzerodivisor.
\end{proof}

\begin{lemma}
  \label{lem:j-saturation}

  \[\left(\sum_{[\sigma]\in \widetilde{\Delta}} J_{\sigma}\right):B^{\infty}=I\]
  
\end{lemma}
\begin{proof}

  Let us start by setting $J=\sum_{[\sigma]\in \widetilde{\Delta}} J_{\sigma}$. We will proceed by showing that $I\cap B^{[k]}\subset J \subset I$. Since $I$ is $B$-saturated this will imply $J:B^{\infty}=I$.
  
  We start by showing $J_{\sigma}\subset I$. Since $I$ is $B$-saturated, it suffices to show that $B^{[k]}\cdot J_{\sigma}\subset I$. This can further be simplified to showing that for each element $f\in J_{\sigma}$ and each cone $\gamma\in \Sigma$, we have $x_{\widehat{\gamma}}^kf \in I$.

  For $f\in J_{\sigma}$ there exists $f'\in \widetilde{J}_{\sigma}$ such that $f=x_{\widehat{\sigma\cup \left\{\tau\right\}}}^k f'$. For convenience, let $\gamma'=\gamma\cap (\sigma \cup \left\{\tau\right\})$. Then $\gamma'\in S(\sigma)$ and thus $f'\in I:x_{\widehat{\gamma'}}^{k}$ which implies $x_{\widehat{\gamma'}}^kf' \in I$. But we know
  \[x_{\widehat{\gamma'}}=\lcm\left(x_{\widehat{\gamma}},x_{\widehat{\sigma\cup\left\{\tau\right\}}}\right)\, \Big\vert\, x_{\widehat{\gamma}}x_{\widehat{\sigma\cup\left\{\tau\right\}}},\] and thus \[x_{\widehat{\gamma'}}^kf'\, \Big|\, x_{\widehat{\gamma}}^kx_{\widehat{\sigma\cup\left\{\tau\right\}}}^kf'.\]
  
  Finally since $x_{\widehat{\gamma}}^kf=x_{\widehat{\gamma}}^kx_{\widehat{\sigma\cup\left\{\tau\right\}}}^kf'$ we have $x_{\widehat{\gamma}}^kf\in I$. Thus $J_\sigma\subset I$ which implies that $J\subset I$.

  Next we will show that $I\cap B^{[k]}\subset J$. Since $I\cap B^{[k]}=\sum_{\gamma\in \Sigma} I_{\gamma}$, it suffices to show that $I_{\gamma}\subset J_{\sigma}$ for all $[\sigma]\in \widetilde{\Delta}$ and $\gamma\in S(\sigma)$. Now expanding the definition of $J_{\sigma}$, we see that it suffices to show for all $\omega\in S(\sigma)$ that $I_{\gamma}\subset x_{\widehat{\sigma\cup\left\{\tau\right\}}}^k (I:x_{\widehat{\omega}}^\infty)$. Finally, recall that $I_{\gamma}= I\cap \left<x_{\widehat{\gamma}}^k\right>\subset \left<x_{\widehat{\sigma\cap\left\{\tau\right\}}}^{k}\right>$. Thus this is equivalent to the following statement: \[I_{\gamma}:x_{\widehat{\sigma\cap\left\{\tau\right\}}}^{k}\subset I:x_{\widehat{\omega}}^\infty.\]
  
  Now let $f\in I_{\gamma}:x_{\widehat{\sigma\cap\left\{\tau\right\}}}^{k}$ and consider $x_{\widehat{\omega}}^kf$. Since $\omega\in S(\sigma)$, we can expand to get the following: \[x_{\widehat{\omega}}^kf= x_{\sigma\cup\left\{\tau\right\}\setminus \omega}^kx_{\widehat{\sigma\cup \left\{\tau\right\}}}^kf.\]

  Finally, since $x_{\widehat{\sigma\cup \left\{\tau\right\}}}^kf\in I_{\gamma}\subset I$, we get $x_{\widehat{\omega}}^kf\in I$. Thus we have that $I_{\gamma}\subset x_{\widehat{\sigma\cup\left\{\tau\right\}}}^k\cdot (I:x_{\widehat{\omega}}^k)$. As stated above, this implies $I\cap B^{[k]}\subset J$.

  Together with the previous work, this gives the inclusions $I\cap B^{[k]}\subset J \subset I$ and thus $J:B^{\infty} = I$.
\end{proof}

At this point we move back to Theorem~\ref{thm:main-theorem}, and combine the results of the previous lemmas to show that $\widetilde{\Delta}$ gives a virtual resolution of $S/I$.

\begin{proof}[Proof of Theorem~\ref{thm:main-theorem}]

  We will proceed to show that $\pdim S/J\leq n$ for \[J:=\sum_{\substack{[\sigma]\in\widetilde{\Delta} \\ \dim \sigma=n}} J_{\sigma}.\] Assuming the conditions of Corollary~\ref{cor:resolution} and applying it to $\widetilde{\Delta}$, the following complex will be acyclic:
  \begin{equation}\label{eq:j-complex}
    0\rightarrow \bigoplus_{\dim\sigma = 1} J_{\sigma}\rightarrow \cdots \rightarrow \bigoplus_{\dim\sigma = n-1} J_{\sigma}\rightarrow \bigoplus_{\dim\sigma = n} J_{\sigma}\rightarrow J\rightarrow 0.
  \end{equation}

  For Corollary~\ref{cor:resolution}, we need that that for $\sigma\subset \gamma$ we have $J_{\sigma}\subset J_{\gamma}$ and that every subcomplex $\widetilde{\Delta}_m$ induced by a monomial is contractible. The first is a direct consequence of Lemma~\ref{lem:j-intersect}.

  
  To show that $\widetilde{\Delta}_m$ is contractible for all monomials $m$ we will split into two cases, either $\widetilde{\Delta}_m$ contains two cells not contained in any common face, or every pair of cells is contained in a common face.\\

  \noindent \emph{Case 1:} $\widetilde{\Delta}_m$ contains two cells that are not contained in any common face in $\widetilde{\Delta}$.

  Here we will prove that $\widetilde{\Delta}_m=\widetilde{\Delta}$ and thus contractible. So let $[\sigma], [\gamma]\in\widetilde{\Delta}$ be two cells not contained in any common face. This implies that $\sigma\cap\gamma=\emptyset$. Since $J_{\sigma}\subset \left<x_{\widehat{\tau\cup \sigma}}^k\right>$ we get $J_{\sigma}\cap J_{\gamma}\subset \left<x_{\widehat{\tau}}^k\right>$. On the other hand, $\left<x_{\widehat{\tau}}^k\right>\subset J_{\mu}$ for all cells $[\mu]\in\widetilde{\Delta}$ and thus $m\in J_{\mu}$ for all cells $[\mu]$. This gives that $[\mu]\in \widetilde{\Delta}_m$ for every cell in $\widetilde{\Delta}$ and thus $\widetilde{\Delta}_m=\widetilde{\Delta}$ and since $\widetilde{\Delta}$ is contractible $\widetilde{\Delta}_m$ is contractible as well.\\

  \noindent \emph{Case 2:} All pairs of cells in $\widetilde{\Delta}_m$ are contained in a common face of $\widetilde{\Delta}$.

  We will now show that $\widetilde{\Delta}_m$ is the closure of a cell in $\widetilde{\Delta}$. Let $[\sigma],[\gamma]$ be two cells in $\widetilde{\Delta}_m$. By assumption $\sigma\cap\gamma\neq \emptyset$. Since $\Sigma$ is simplicial $\sigma\cap\gamma$ is a cone. Furthermore, $\sigma\cap\tau$ does not contain $\tau$ and thus $[\sigma\cap\gamma]$ is a cell of $\widetilde{\Delta}$. Moreover, $[\sigma]$ and $[\gamma]$ are in the closure of $[\sigma\cap\gamma]$. By assumption $m\in J_{\sigma}$ and $m\in J_{\gamma}$ thus $m\in J_{\sigma\cap\gamma}$. Thus $[\sigma\cap\gamma]$ is a cell of $\widetilde{\Delta}_m$. Then iteratively applying this procedure to the cells in $\widetilde{\Delta}_m$, we find that $\widetilde{\Delta}_m$ must be the closure of a single cell. Then since every cell in $\Delta$ and thus $\widetilde{\Delta}$ has a contractible closure, $\widetilde{\Delta}_m$ is contractible.

  Now that we have shown $\widetilde{\Delta}_m$ is contractible, we need to show that the resulting resolution is sufficiently short. By Lemma~\ref{lem:j-pdim} we have $\pdim J_{\sigma}\leq \dim \sigma - 1$. Then as in the proof of Lemma~\ref{lem:short} we construct the total complex of the double complex given by resolving each of the terms in the complex in Equation~\ref{eq:j-complex} and find that $\pdim J\leq n-1$ and thus $\pdim S/J\leq n$.
\end{proof}

\begin{remark}
  The procedure in the proof depends on the ability to construct a correspondence between the cells of $\widetilde{\Delta}$ and $\Delta$. and as such, we cannot repeat the procedure to give even shorter resolutions.
\end{remark}

\begin{example}
  Continuing from Example~\ref{ex:labeling}, we can view the process in the proof of Theorem~\ref{thm:main-theorem} as ``collapsing'' the labeled cell complex from the example. The information from the cells that are removed from the complex are in some sense spread across the remaining cells. In this example we use the ray corresponding to $x_0$ for $\tau$. Thus the vertices that remain in $\widetilde{\Delta}$ correspond to the cones $\left\{x_1,x_2,x_4\right\}$, $\left\{x_1,x_3,x_4\right\}$, and $\left\{x_2,x_3,x_4\right\}$. This yields the following labeled cell complex:
  \begin{center}
    \begin{tikzpicture}
      \coordinate (w1) at (1,0);
      \coordinate (w2) at (-0.7071,0.7071);
      \coordinate (w3) at (-0.7071,-0.7071);
      
      \node[right] at (w1){$x_2^6\cdot \left<{x}_{3},{x}_{1}^{4} {x}_{4},{x}_{1}^{5}\right>$};
      \node[below] at (w3){$x_3^6\cdot \left<{x}_{2} {x}_{1},{x}_{4}^{2},{x}_{1}^{5}\right>$};
      \node[above] at (w2){$x_1^6\cdot \left<{x}_{2} {x}_{3},{x}_{2}^{2},{x}_{0}^{2},{x}_{3}^{3} \right> $};
      \draw (w1)--(w2)--(w3)--cycle;
    \end{tikzpicture}.
  \end{center}

  Then let
  \begin{align*}
    J_{1}&=x_1^6\cdot \left<{x}_{2} {x}_{3},{x}_{2}^{2},{x}_{0}^{2},{x}_{3}^{3} \right>,\\
    J_{2}&=x_2^6\cdot \left<{x}_{3},{x}_{1}^{4} {x}_{4},{x}_{1}^{5}\right>,\\
    J_{3}&=x_3^6\cdot \left<{x}_{2} {x}_{1},{x}_{4}^{2},{x}_{1}^{5}\right>.
  \end{align*}

  Then for $S=k[\Sigma]$, this gives the following resolutions:
  \[0\rightarrow S^2 \rightarrow S^5 \rightarrow S^4 \rightarrow J_1,\]
  \[0\rightarrow S^1 \rightarrow S^3 \rightarrow S^3 \rightarrow J_2,\]
  \[0\rightarrow S^1 \rightarrow S^3 \rightarrow S^3 \rightarrow J_3.\]

  Furthermore the ideals $J_{1}\cap J_{2},J_{1}\cap J_{2},J_{1}\cap J_{2},J_{1}\cap J_{2}\cap J_{3}$ are all principal and thus free of rank $1$. Then the resulting free resolution from the total complex of the cellular resolution given by the above diagram is

  \[0\rightarrow S^{5}\rightarrow S^{14}\rightarrow S^{10}\rightarrow J.\]

  Thus as desired $\pdim S/J \leq 3$, and so $\vpdim S/J\leq \dim \Sigma$. In this instance the resolution this procedure yields happens to be a minimal free resolution of $J$.
\end{example}

\end{document}